\documentclass[letterpaper,12pt]{article}

\usepackage{amsmath}
\allowdisplaybreaks

\usepackage{microtype,booktabs}
\usepackage{amsmath,amssymb,amsfonts,amsthm,bm}
\usepackage{authblk}
\usepackage{mathrsfs,dsfont}
\usepackage{indentfirst}
\usepackage[pdftex]{color,graphicx}
\usepackage[pdftex,bookmarks,unicode,colorlinks,linkcolor=blue]{hyperref}
\usepackage{tikz}
\usepackage{makecell}
\usepackage{subfigure}
\usetikzlibrary{arrows, decorations.pathmorphing, backgrounds, positioning, fit, petri, automata,cd}
\usepackage{rotating}
\usepackage{orcidlink}
\numberwithin{equation}{section}

\title{\Large Onsager--Machlup functionals for McKean--Vlasov SDEs via Euler-type approximation}

\date{}
\author[$\dag$]{\small Ying Chao $^\ast$}
\author[$\ddag$]{\small Pingyuan Wei \thanks{Corresponding authors\par \textit{\;\;Email:} yingchao1993@xjtu.edu.cn (Ying Chao),  pwei@seu.edu.cn (Pingyuan Wei)}}
\affil[$\dag$]{\footnotesize School of Mathematics and Statistics, Xi'an Jiaotong University, Xi'an, Shaanxi 710049, China.}
\affil[$\ddag$]{\footnotesize School of Mathematics, Southeast University, Nanjing 211189, China.}

\newtheorem{assumption}{Assumption}[section]
\newtheorem{theorem}{Theorem}[section]

\newtheorem{lemma}[theorem]{Lemma}
\newtheorem{pro}[theorem]{Proposition}

\theoremstyle{remark}
\newtheorem{remark}[theorem]{Remark}
\theoremstyle{definition}
\newtheorem{definition}[theorem]{Definition}

\begin{document}
\maketitle

\begin{abstract}
The Onsager--Machlup action functional provides a variational framework for characterizing the most probable transition paths of stochastic systems and plays an important role in the study of nonequilibrium fluctuations. Its extension to McKean--Vlasov stochastic differential equations is complicated by the intrinsic distribution dependence of the coefficients. In this paper, we address this difficulty by introducing an Euler-type approximation scheme based on classical, distribution-free stochastic differential equations. Combining the classical Onsager--Machlup theory with a convergence argument for the approximation sequence, we derive an explicit expression for the Onsager--Machlup functional associated with the McKean--Vlasov SDE. The proposed approach is constructive and extends to a broad class of distribution-dependent stochastic systems.\\


\end{abstract}


\renewcommand{\theequation}{\thesection.\arabic{equation}}
\setcounter{equation}{0}

\section{Introduction}

In stochastic dynamical systems, transitions between distinct metastable states (commonly known as metastable transitions) often emerge from the interplay between nonlinear dynamics and random perturbations. Such phenomena are observed in a wide range of contexts, including climate change (e.g., \cite{Ditlevsen1999observation}), nucleation in phase transitions (e.g., \cite{Heymann2008Pathways}), conformational switching of macromolecules (e.g., \cite{Weinan2002string}), and gene transcription (e.g., \cite{Zheng2016transitions}). Beyond these domain-specific settings, they lead to a common mathematical problem: given two metastable states, which trajectory connecting them is the most probable? A powerful framework for addressing this question is provided by the Onsager–Machlup (OM) action functional, which quantifies the likelihood of such transitions and identifies the most probable transition pathway (MPTP) as its minimizer \cite{Brocker2019correct}.

The OM action functional originated in the seminal work of Onsager and Machlup \cite{Onsager1953} on irreversible thermodynamic fluctuations, where transition probabilities for Gaussian systems were characterized through a path-integral formalism. Their original theory was restricted to diffusion processes with linear drift and constant diffusion coefficients. It was subsequently extended to nonlinear stochastic systems by Tisza and Manning \cite{tisza1957fluctuations}. Later, rigorous mathematical foundations were established in \cite{Durr1978,Ikeda2014stochastic}, where the OM functional was derived for reversible diffusions and general It\^{o} stochastic differential equations (SDEs), respectively, by analyzing the asymptotic probability that sample paths remain within shrinking tubes around smooth reference curves. In this framework, the OM functional arises as the corresponding limiting action. More recently, the theory has been generalized to non-Gaussian systems. In particular, Bardina et al. \cite{Bardina2002asymptotic} studied tube probabilities for linear SDEs driven by Poisson noise, while subsequent works \cite{Chao2019,Chao2025,Huang2025} extended the framework to broader classes of L\'evy-type models.

Two classical approaches are commonly used to derive the OM functional: the formal path-integral approach \cite{Hunt1981} and the measure-change approach based on the Girsanov theorem \cite{Durr1978,Ikeda2014stochastic}. The former begins with a time discretization of the underlying dynamics, constructs finite-dimensional transition densities, and then passes to the continuum limit. Formally, this leads to a path-integral representation of the transition density from $x_0$ to $x_T$ as a weighted superposition over all possible paths,
\begin{align*}
p_T(x_0, x_T) = &\lim_{n \to \infty} \int_{\mathbb{R}^d} \cdots \int_{\mathbb{R}^d} p_{t_1 - t_0}(x_0, x_1) \cdots p_{t_n - t_{n-1}}(x_{n-1}, x_T) \, dx_1 \cdots dx_{n-1}\\
=& \int_{\phi \in C_{x_0, x_T}[0,T]} \exp\!\bigl( - \mathcal{L}(\phi, \dot{\phi}) \bigr) \, \mathcal{D}(\phi),
\end{align*}
where $C_{x_0,x_T}[0,T]$ denotes the space of continuous paths with fixed endpoints $x_0$ and $x_T$, $\mathcal D$ is the corresponding path measure, and $\mathcal L(\phi,\dot\phi)$ is the OM functional assigning a weight to each path. The second approach characterizes the probability that the process $X_t$ remains within a tube of radius $\varepsilon$ around a smooth reference path $\phi$, and then examines the limit as $\varepsilon \to 0$. More precisely,
$$
\mathds{P}\bigl( \sup_{t \in [0,T]} |X_t - \phi_t| \leq \varepsilon \bigr) \approx C(\varepsilon, T) \exp\!\bigl( - \mathcal{L}(\phi, \dot{\phi}) \bigr),
$$
where $X_t$ is the solution of a diffusion or jump-diffusion SDE, and $C(\varepsilon,T)$ is a positive prefactor depending on $\varepsilon$ and $T$. To compute the probability on the left-hand side, a change of measure via the Girsanov theorem is applied. The central task is then to identify the correct leading-order exponent in the limit as $\varepsilon \to 0$.

The situation becomes significantly more delicate for McKean--Vlasov SDEs. These equations, first introduced by McKean \cite{Mckean1966} in connection with the kinetic theory of Kac \cite{Kac1956}, are characterized by coefficients depending not only on the current state but also on the probability distribution of the solution. They are also referred to as mean-field SDEs, since they arise as the limiting dynamics of weakly interacting particle systems under the propagation of chaos \cite{Sznitman1992}. Owing to their broad relevance in stochastic control, mean-field games, and statistical physics, McKean--Vlasov SDEs have been extensively studied in recent years, including well-posedness \cite{Li2022strong,MR5056525}, ergodicity \cite{MR4291453}, and random attractors \cite{shi2024invariant,cheng2025random}. In particular, Ren and Wang \cite{ren2020space} studied the path-independence of additive functionals for McKean--Vlasov SDEs and established an It\^{o} formula for distribution-dependent functionals involving the $L$-derivative introduced by Lions \cite{cardaliaguet2013notes}.

Despite this progress, the OM theory for McKean--Vlasov SDEs remains far from complete. This is due to several fundamental difficulties caused by the distribution dependence of the coefficients, both in the path-integral approach and in the Girsanov-transform approach. On the one hand, in the path-integral framework, the transition density of a typical McKean--Vlasov SDE is highly nontrivial and does not admit a tractable path representation analogous to the classical diffusion case. On the other hand, in the Girsanov-transform approach, one must identify the correct leading-order exponential term in the vanishing-tube limit as $\varepsilon \to 0$, which becomes technically delicate once the Radon--Nikodym derivative is reformulated via It\^{o}'s formula. Liu et al. \cite{liu2023onsager} investigated this second approach for McKean--Vlasov SDEs driven by additive Brownian noise, under additional assumptions on the drift coefficient, which restricts the scope of their result. To the best of our knowledge, only a few works have addressed this problem \cite{liu2023onsager2}, and most of them are limited to additive Brownian noise or to special forms of the drift. In the present work, we aim to fill this gap by rigorously deriving an Onsager--Machlup functional for a broad class of McKean--Vlasov stochastic systems.

Interacting stochastic particle systems provide a natural approximation to the distribution dependence in McKean--Vlasov SDEs and are widely used in numerical simulations (e.g., \cite{Li2022strong}). However, these systems are inherently $N$-dimensional. Although each finite particle system consists of classical SDEs, the exponent in the corresponding OM functional is typically a sum of $N$ coupled terms, and it is unclear whether this representation admits a closed-form limit as $N \to \infty$. Instead of using particle systems as an intermediate approximation, we adopt an Euler-type scheme composed of classical, distribution-free SDEs. This strategy is inspired by our previous work \cite{MR5056525}, where a similar approximation scheme was used to establish well-posedness under merely locally Lipschitz conditions in the state variable, allowing for superlinear drift growth; related ideas also appear in \cite{Kloeden2010,Li2022strong}. Each approximating system admits an explicit OM functional in closed form. By establishing the convergence of the approximating solutions to the McKean--Vlasov equation, we rigorously derive the OM functional for the original system and formulate the associated variational problem. The proposed framework applies to a broad class of McKean--Vlasov SDEs driven by Brownian noise and can be naturally extended to L\'evy-driven systems.

The remainder of this paper is organized as follows. Section~\ref{sec:2} introduces the notation, assumptions, and basic definitions, and constructs the interpolated Euler-like approximation sequence. Section~\ref{sec:3} develops the general theory. In particular, Section~\ref{sec:31} derives the OM functional for a class of McKean--Vlasov SDEs with multiplicative Brownian noise, while Section~\ref{sec:32} presents a Lagrangian formulation for analyzing the most probable transition pathways between arbitrary states. The proof of the main theorem, based on the interpolated Euler-like approximation, is provided in Section~\ref{sec:33}. Section~\ref{sec:4} illustrates the theory through a concrete example, and Section~\ref{sec:5} concludes the paper with a brief summary and discussion.

\section{Preliminaries}\label{sec:2}

\subsection{Basic notation and framework}

We begin by introducing some notations used throughout the paper. Let $|\cdot|$ and $\langle \cdot, \cdot \rangle$ denote the Euclidean norm and the inner product on $\mathbb{R}^d$, respectively. For a matrix $A$, we use the Frobenius norm $\|A\| := \sqrt{\operatorname{tr}(AA^\top)}$. Denote by $\mathcal{P}_2(\mathbb{R}^d)$ the space of probability measures on $\mathbb{R}^d$ with finite second moment, i.e.,
\[
\mathcal{P}_2(\mathbb{R}^d) := \left\{ \mu : \mu \text{ is a probability measure on } \mathbb{R}^d \text{ and } \mu(|\cdot|^2) := \int_{\mathbb{R}^d} |x|^2 \, \mu(dx) < \infty \right\}.
\]
Note that for any $x \in \mathbb{R}^d$, the Dirac delta measure $\delta_x$ belongs to $\mathcal{P}_2(\mathbb{R}^d)$. Moreover, if $\phi_t$ is a deterministic path, then the law of $\phi_t$ is the Dirac measure at $\phi_t$, i.e., $\mathscr{L}_{\phi_t} = \delta_{\phi_t}$. The space $\mathcal{P}_2(\mathbb{R}^d)$ is a Polish space equipped with the $L^2$-Wasserstein distance
\[
W_2(\mu_1, \mu_2) := \inf_{\pi \in \mathscr{C}(\mu_1, \mu_2)} \left( \int_{\mathbb{R}^d \times \mathbb{R}^d} |x-y|^2 \, \pi(dx, dy) \right)^{1/2}, \;\; \mu_1, \mu_2 \in \mathcal{P}_2(\mathbb{R}^d),
\]
where $\mathscr{C}(\mu_1, \mu_2)$ denotes the set of all probability measures on $\mathbb{R}^d \times \mathbb{R}^d$ with marginals $\mu_1$ and $\mu_2$, respectively. In particular, if $\mu_1 = \mathscr{L}_X$ and $\mu_2 = \mathscr{L}_Y$ are the laws of random variables $X$ and $Y$, then
\[
\bigl( W_2(\mu_1, \mu_2) \bigr)^2 \leq \int_{\mathbb{R}^d \times \mathbb{R}^d} |x-y|^2 \, \mathscr{L}_{(X,Y)}(dx, dy) = \mathbb{E}|X-Y|^2,
\]
where $\mathscr{L}_{(X,Y)}$ denotes the joint distribution of $(X,Y)$. 

Given $T > 0$, let $C([0,T]; \mathbb{R}^d)$ denote the space of continuous functions from $[0,T]$ to $\mathbb{R}^d$ endowed with the supremum norm, and denote by $C^{x_0}([0,T]; \mathbb{R}^d)$ the subspace of $C([0,T]; \mathbb{R}^d)$ consisting of functions satisfying $\phi(0) = x_0$. For $1 \leq p < \infty$, we write $L^p(\Omega; \mathbb{R}^d)$ for the space of $\mathbb{R}^d$-valued random variables $X$ such that $\mathbb{E}|X|^p < \infty$. Similarly, let $L^p(\Omega; C([0,T]; \mathbb{R}^d))$ denote the space of $C([0,T]; \mathbb{R}^d)$-valued random variables $X$ satisfying $\mathbb{E}\bigl[\sup_{0 \le t \le T} |X(t)|^p\bigr] < \infty$. Equipped with the norm
\[
\|X\|_{L^p} := \left( \mathbb{E}\left[ \sup_{0 \le t \le T} |X(t)|^p \right] \right)^{1/p},
\]
$L^p(\Omega; C([0,T]; \mathbb{R}^d))$ is a Banach space.

For readability, we also recall the notion of the $L$-derivative, also known as the Lions derivative, i.e., a derivative with respect to probability measures; for more details, see \cite{cardaliaguet2013notes,ren2020space}. Let $\mathrm{Id}: \mathbb{R}^d \to \mathbb{R}^d$ denote the identity map, $\mathrm{Id}(x)=x$ for $x\in\mathbb{R}^d$. Note that for any $\mu \in \mathcal{P}_2(\mathbb{R}^d)$ and $\psi \in L_\mu^2(\mathbb{R}^d; \mathbb{R}^d)$, we have $\mu \circ (\mathrm{Id} + \psi)^{-1} \in \mathcal{P}_2(\mathbb{R}^d)$, where $L_\mu^2(\mathbb{R}^d; \mathbb{R}^d)$ denotes the space of square-integrable vector-valued functions from $\mathbb{R}^d$ to $\mathbb{R}^d$ with respect to $\mu$.

\begin{definition}[$L$-derivative]\label{L-def}
A function $h:\mathcal{P}_2(\mathbb{R}^d)\to \mathbb{R}$ is said to be \emph{$L$-differentiable at $\mu\in \mathcal{P}_2(\mathbb{R}^d)$} if the map
$$
L_\mu^2(\mathbb{R}^d; \mathbb{R}^d)\ni \psi
\longmapsto
h\bigl(\mu\circ(\mathrm{Id}+\psi)^{-1}\bigr)
$$
is Fr\'echet differentiable at $\psi=0$; i.e., there exists a unique $\xi\in L_\mu^2(\mathbb{R}^d; \mathbb{R}^d)$ such that
$$
\lim_{\mu(|\psi|^2)\to 0}
\frac{
h\bigl(\mu\circ(\mathrm{Id}+\psi)^{-1}\bigr)-h(\mu)-\mu(\langle \xi,\psi\rangle)
}{
\sqrt{\mu(|\psi|^2)}
}
=0.
$$
In this case, we write
$$
D_\mu h(\mu)=\xi,
$$
and call $D_\mu h(\mu)$ the \emph{$L$-derivative} of $h$ at $\mu$. The function $h$ is said to be \emph{$L$-differentiable on $\mathcal{P}_2(\mathbb{R}^d)$} if $D_\mu h(\mu)$ exists for every $\mu\in \mathcal{P}_2(\mathbb{R}^d)$.
\end{definition}


With these notations in place, we now consider the following $d$-dimensional McKean--Vlasov SDE on the interval $[0,T]$:
\begin{equation}\label{Main-equation}
dX_t = b\bigl(t, X_t, \mu_t\bigr) \, dt + \sigma\bigl(t, X_t, \mu_t\bigr) \, dW_t, \qquad X(0) = x_0,
\end{equation}
where $\mu_t$ denotes the law of $X_t$ at time $t$, and $\{W_t\}_{t \geq 0}$ is an $m$-dimensional standard Wiener process defined on a complete filtered probability space with filtration $(\mathscr{F}_t)_{t \geq 0}$ satisfying the usual conditions. The coefficients
 $b: [0,T] \times \mathbb{R}^d \times \mathcal{P}_2(\mathbb{R}^d) \to \mathbb{R}^d$, $\sigma: [0,T] \times \mathbb{R}^d \times \mathcal{P}_2(\mathbb{R}^d) \to \mathbb{R}^{d \times m}$
are Borel measurable. 

Throughout the paper, we impose the following assumptions on \eqref{Main-equation}.

\begin{assumption}\label{assump:main}
There exists a constant $\kappa \geq 2$ such that the following conditions hold.
\begin{enumerate}
\item {(One-sided locally Lipschitz condition on the state variable)} 
For every $R > 0$, there exists a positive constant $L_R$ such that for any $t \in [0,T]$, $x, y \in \mathbb{R}^d$ with $|x| \vee |y| \leq R$, and $\mu \in \mathcal{P}_2(\mathbb{R}^d)$,
\begin{equation*}
\langle x - y, b(t,x,\mu) - b(t,y,\mu) \rangle \vee \|\sigma(t,x,\mu) - \sigma(t,y,\mu)\|^2  \leq L_R |x-y|^2,\notag
\end{equation*}
where the symbol ``$\vee$'' denotes the maximum of the multiple terms.

\item {(Globally Lipschitz condition on the measure)} 
There exists a positive constant $L$ such that for any $t \in [0,T]$, $x \in \mathbb{R}^d$, and $\mu_1, \mu_2 \in \mathcal{P}_2(\mathbb{R}^d)$,\notag
\begin{equation*}
|b(t,x,\mu_1) - b(t,x,\mu_2)|^2 + \|\sigma(t,x,\mu_1) - \sigma(t,x,\mu_2)\|^2  \leq L \, W_2^2(\mu_1, \mu_2).\notag
\end{equation*}

\item {(Continuity)} 
For each $t \in [0,T]$, the maps 
$b(t,\cdot,\cdot)$ and $\sigma(t,\cdot,\cdot)$ are continuous on $\mathbb{R}^d \times \mathcal{P}_2(\mathbb{R}^d)$.\notag

\item {(One-sided linear and global linear growth condition)} 
There exists a positive constant $K$ such that for any $t \in [0,T]$, $x \in \mathbb{R}^d$, and $\mu \in \mathcal{P}_2(\mathbb{R}^d)$,
\begin{equation*}
\langle x, b(t,x,\mu) \rangle \vee \|\sigma(t,x,\mu)\|^2 \leq K \bigl(1 + |x|^2 + W_2^2(\mu, \delta_0)\bigr).\notag
\end{equation*}

\item {($\kappa$-order growth condition on the drift coefficient)} 
There exists a positive constant $K_1$ such that for any $t \in [0,T]$, $x \in \mathbb{R}^d$, and $\mu \in \mathcal{P}_2(\mathbb{R}^d)$,
\begin{equation*}
|b(t,x,\mu)|^2 \leq K_1 \bigl(1 + |x|^\kappa + W_2^\kappa(\mu, \delta_0)\bigr).\notag
\end{equation*}
\end{enumerate}
\end{assumption}
As discussed in the Introduction, the distribution dependence of the coefficients creates substantial difficulties in the analysis of McKean--Vlasov SDEs. To address this issue, we introduce an interpolated Euler-type approximation scheme built from classical SDEs with frozen laws. This approximation will be used to establish convergence of the approximating sequence to the solution of \eqref{Main-equation}, and ultimately to derive the Onsager--Machlup functional.

\subsection{Interpolated Euler-like sequence and convergence result}\label{subsec:euler-sequence}

Let us construct the following Euler-like sequence for the McKean--Vlasov SDE \eqref{Main-equation}. Given $T > 0$, for $n \geq 1$ sufficiently large such that $h_n := T/n \in (0, 1]$, define $t_k^n := k h_n$ for $k = 0, 1, \dots, n$, and set
\(
\kappa_n(t) := \sup\bigl\{ s \in \{0, h_n, 2h_n, \dots, n h_n\} : s \leq t \bigr\}\) for all $t\in[0,T]$. We now construct $X^{(n)}_t$ step by step on the intervals $[0, t_1^n]$, $(t_1^n, t_2^n]$, $\dots$, $(t_{n-1}^n, T]$.

Let $X^{(n)}(0) = x_0$ and for $t \in [0, t_1^n]$, consider the following SDE
\begin{equation}\label{eq:first-step}
dX^{(n)}_t = b\bigl(t, X^{(n)}_t, \mu_0^{(n)}\bigr) \, dt + \sigma\bigl(t, X^{(n)}_t, \mu_0^{(n)}\bigr) \, dW_t,
\end{equation}
where $\mu_0^{(n)} = \mathcal{L}_{X^{(n)}_0}$ denotes the law of $X^{(n)}_0$. This is a classical SDE independent of the law of $X^{(n)}_t$. Under Assumption \ref{assump:main} (1) and (4), for any $t \in [0, t_1^n]$ and $x, y \in \mathbb{R}^d$ with $|x| \vee |y| \leq R$,
\begin{align}
&\bigl\langle x - y, b(t,x,\mu_0^{(n)}) - b(t,y,\mu_0^{(n)}) \bigr\rangle + \bigl\| \sigma(t,x,\mu_0^{(n)}) - \sigma(t,y,\mu_0^{(n)}) \bigr\|^2 \leq L_R |x-y|^2, \notag\\
&\bigl\langle x, b(t,x,\mu_0^{(n)}) \bigr\rangle + \bigl\| \sigma(t,x,\mu_0^{(n)}) \bigr\|^2 \leq 2K \bigl(1 + \mathbb{E}|X^{(n)}_0|^2\bigr) \bigl(1 + |x|^2\bigr).\notag
\end{align}
These conditions imply that \eqref{eq:first-step} admits a unique solution on $[0, t_1^n]$; see, e.g., 
Theorem 3.1.1, p.44 of \cite{Prevot2007}. Moreover, by Assumption \ref{assump:main} (5), for any $p>0$, there exists a positive constant $C$ such that
\begin{equation}
\mathbb{E}\left[ \sup_{0 \leq t \leq t_1^n} |X^{(n)}_t|^p \right] \leq C \bigl(1 + \mathbb{E}|X^{(n)}_0|^p\bigr),\notag
\end{equation}
the proof of which is similar to that of Lemma 1 in \cite{MR5056525}. Consequently, we can define $X^{(n)}_{t_1^n}$ (which satisfies $\mathbb{E}|X^{(n)}_{t_1^n}|^p < \infty$) and $\mu_{t_1^n}^{(n)} := \mathcal{L}_{X^{(n)}_{t_1^n}}$.

For $k = 1$ and $t \in (t_1^n, t_2^n]$, we replace $(X^{(n)}_0, \mu_0^{(n)})$ with $(X^{(n)}_{t_1^n}, \mu_{t_1^n}^{(n)})$ and obtain analogous results. Repeating this procedure yields that for any $\beta > 0$ and $x_0 \in L^{\beta}(\Omega; \mathbb{R}^d)$, the approximating SDE
\begin{equation}\label{eq:euler-approx}
dX^{(n)}_t = b\bigl(t, X^{(n)}_t, \mu_{\kappa_n(t)}^{(n)}\bigr) \, dt + \sigma\bigl(t, X^{(n)}_t, \mu_{\kappa_n(t)}^{(n)}\bigr) \, dW_t
\end{equation}
admits a global solution on $[0,T]$, where $\mu_{\kappa_n(t)}^{(n)}$ denotes the law of $X_{\kappa_n(t)}^{(n)}$.

With the Euler-like sequence constructed above, we have the following convergence result, which establishes the well-posedness of \eqref{Main-equation}.

\begin{pro}[Approximation convergence and well-posedness] \label{thm:convergence}
Let Assumption~\ref{assump:main} hold, and suppose further that $x_0 \in L^{\beta}(\Omega; \mathbb{R}^d)$ for any $\beta > 0$. Then the sequence $\{X^{(n)}\}_{n \geq 1}$ of solutions to \eqref{eq:euler-approx} converges in $L^2(\Omega; C([0,T]; \mathbb{R}^d))$ to a limit $X$, which is the unique strong solution of the McKean--Vlasov equation \eqref{Main-equation}.
\end{pro}

Proposition~\ref{thm:convergence} is a special case of Theorem 3.1 in our recent work \cite{MR5056525}, where the jump measure is set to zero. In that work, we proved that the Euler-like sequence is Cauchy in the complete space $L^{\kappa}(\Omega; C([0,T]; \mathbb{R}^d))$ for $\kappa \geq 2$, and the limiting process $X$ is identified as the unique strong solution of \eqref{Main-equation}. The details are therefore omitted here for brevity.

\subsection{Onsager--Machlup functional}

For the McKean--Vlasov SDE \eqref{Main-equation}, we adopt the Onsager--Machlup viewpoint and introduce the following definition.

\begin{definition}\label{def:om}
(OM function $\&$ OM action functional)  Let $\varepsilon > 0$ be given and consider a tube surrounding a reference path $\phi \in C^{x_0}([0,T]; \mathbb{R}^d)$. If, for sufficiently small $\varepsilon$, the probability that the solution process $X$ of \eqref{Main-equation} lies within this tube admits the asymptotic estimate
$$
\mathds{P}\bigl(\{\|X - \phi\| \leq \varepsilon\}\bigr) \propto C(\varepsilon, T) \exp\left\{ -\int_0^T \mathcal{OM}(\phi, \dot{\phi}) \, dt \right\},
$$
then the integrand $\mathcal{OM}(\phi, \dot{\phi})$ is called the \textit{Onsager--Machlup function}. Here, $\propto$ denotes proportionality as $\varepsilon \to 0$, and $C(\varepsilon, T) = \mathds{P}\bigl(\{ \sup_{0\leq t \leq T}|W_t| \leq \varepsilon\}\bigr)$ is the probability of Brownian motion in the $\varepsilon$-tube independent of $\phi$. The quantity
$$
I(\phi, \dot{\phi}) := \int_0^T \mathcal{OM}(\phi, \dot{\phi}) \, dt
$$
is called the \textit{Onsager--Machlup action functional}.
\end{definition}

The tube probability $\mathds{P}\bigl(\{\|X-\phi\|\leq \varepsilon\}\bigr)$ depends sensitively on the choice of the reference path $\phi$. Once the Onsager--Machlup action functional $I(\phi,\dot{\phi})$ is determined, the most probable transition pathway can be characterized as a minimizer of this functional, that is, as a solution of the associated variational problem. The resulting minimizer $\phi$ then represents the  most probable transition pathway (MPTP) of \eqref{Main-equation}.

Note that the system \eqref{Main-equation} under consideration is driven by multiplicative noise. In this case, the diffusion intensity depends on the state variable, so the ambient Euclidean geometry is no longer compatible with the local fluctuation structure of the process. It is therefore natural, and indeed necessary, to introduce a Riemannian metric induced by the diffusion tensor \cite{Zeitouni1988,Capitaine2000,Grong2024}. More precisely, define the diffusion matrix (viewed as a covariant tensor) by
\begin{equation}\label{eq:diffusion-tensor}
\Sigma(t,\phi_t,\delta_{\phi_t})
:=
\sigma(t,\phi_t,\delta_{\phi_t})
\sigma(t,\phi_t,\delta_{\phi_t})^\top
=
\sum_{\alpha=1}^m
\sigma_\alpha(t,\phi_t,\delta_{\phi_t})
\otimes
\sigma_\alpha(t,\phi_t,\delta_{\phi_t}),
\end{equation}
where $\sigma_\alpha$ denotes the $\alpha$-th column of the diffusion matrix $\sigma$. Assume that $\Sigma(t,\phi_t,\delta_{\phi_t})$ is positive definite. We then equip $\mathbb{R}^d$ with the induced Riemannian metric
\begin{equation}\label{eq:Riemannian metric}
\textsl{g}_{ij}(t,x,\delta_{x})
=
\bigl(\Sigma(t,x,\delta_{x})^{-1}\bigr)_{ij},
\qquad
\textsl{g}
=
\sum_{i,j=1}^d
\textsl{g}_{ij}\,dx^i\otimes dx^j.
\end{equation}
The associated Christoffel symbols are given by
$$
\Gamma_{lj}^{i}
=
\frac{1}{2}
\sum_{r=1}^d
\textsl{g}^{ir}
\left(
\frac{\partial \textsl{g}_{lr}}{\partial x^j}
+
\frac{\partial \textsl{g}_{jr}}{\partial x^l}
-
\frac{\partial \textsl{g}_{lj}}{\partial x^r}
\right),
$$
where $(\textsl{g}^{ij})_{1\le i,j\le d}$ denotes the inverse matrix of $(\textsl{g}_{ij})_{1\le i,j\le d}$. This intrinsic geometry reflects the state-dependent covariance structure of the noise and provides the appropriate local notion of distance and energetic cost for path fluctuations. Consequently, the MPTPs should be characterized relative to this Riemannian structure rather than the standard Euclidean one.

\section{Main Results}\label{sec:3}

\subsection{Onsager--Machlup formula for McKean--Vlasov SDEs}\label{sec:31}

We now state our main result associated with Onsager--Machlup  functional for system \eqref{Main-equation}, whose proof will be presented in Section \ref{sec:33}.

\begin{theorem}[Onsager--Machlup formula for McKean--Vlasov SDEs]\label{thm:om}
Consider the McKean--Vlasov SDE \eqref{Main-equation} with initial condition $X_0=x_0\in\mathbb{R}^d$. Let Assumption~\ref{assump:main} hold, and suppose further that $x_0\in L^{\beta}(\Omega;\mathbb{R}^d)$ for every $\beta>0$. Let $\phi\in C^{x_0}([0,T];\mathbb{R}^d)$ be a reference path satisfying $\phi(0)=x_0$. Then, up to an additive constant, the Onsager--Machlup action functional is given by
\begin{equation}\label{eq:om-action}
I(\phi,\dot{\phi})
=
\int_0^T
\mathcal{OM}\bigl(t,\phi_t,\dot{\phi}_t,\delta_{\phi_t}\bigr)\,dt,
\end{equation}
where the Onsager--Machlup function is defined by
\begin{align}\label{OMf}
\mathcal{OM}\bigl(t,\phi_t,\dot{\phi}_t,\delta_{\phi_t}\bigr)
:=
\frac{1}{2}\Big[&
\bigl(\dot{\phi}_t-\widetilde{b}(t,\phi_t,\delta_{\phi_t})\bigr)^\top
\Sigma^{-1}(t,\phi_t,\delta_{\phi_t})
\bigl(\dot{\phi}_t-\widetilde{b}(t,\phi_t,\delta_{\phi_t})\bigr)
\notag\\
&\quad
+\operatorname{div}\widetilde{b}(t,\phi_t,\delta_{\phi_t})
-\frac{1}{6}R(t,\phi_t,\delta_{\phi_t})
\Big].
\end{align}
In the above expression, $\Sigma$ is the positive definite diffusion matrix defined in \eqref{eq:diffusion-tensor}, $R(t,\phi_t,\delta_{\phi_t})$ denotes the scalar curvature associated with the induced Riemannian metric \eqref{eq:Riemannian metric}, and the modified drift $\widetilde{b}(t,\phi_t,\delta_{\phi_t})$ is given componentwise by
\[
\widetilde{b}^{\,i}
=
b^{i}
-\frac{1}{2}\sum_{l,j}\Sigma^{lj}\Gamma_{lj}^{i},
\]
while its Riemannian divergence is
\[
\operatorname{div}\widetilde{b}
=
\frac{1}{\sqrt{\det(\Sigma^{-1})}}
\sum_i
\frac{\partial}{\partial x^i}
\left(
\widetilde{b}^{\,i}\sqrt{\det(\Sigma^{-1})}
\right).
\]
Here $\Gamma_{lj}^{i}$ are the Christoffel symbols of the metric, and $\det(\cdot)$ denotes the determinant of a matrix.
\end{theorem}

\begin{remark}\label{rem:generality}
Theorem~\ref{thm:om} covers in particular the case of additive noise, where the diffusion coefficient is constant (i.e., $\sigma$ is independent of $t$, $x$, and $\mu$). In this situation, the induced Riemannian metric is flat, the curvature term vanishes, and the geometric divergence reduces to the standard Euclidean divergence $\operatorname{div} b$. Consequently, the Onsager--Machlup function takes the simpler form
\begin{equation}\label{eq:om-constant}
\mathcal{OM}\bigl(t, \phi_t, \dot{\phi}_t, \delta_{\phi_t}\bigr)
=
\Bigl|
\frac{\dot{\phi}_t - b(t, \phi_t, \delta_{\phi_t})}{\sigma}
\Bigr|^2
+
\operatorname{div} b(t, \phi_t, \delta_{\phi_t}),
\end{equation}
where $\operatorname{div} b = \sum_{i=1}^d \partial_{x_i} b_i(t, \phi_t, \delta_{\phi_t})$. 
Furthermore, Theorem~\ref{thm:om} is consistent with the classical Onsager--Machlup result \cite{Zeitouni1988} when \eqref{Main-equation} reduces to a standard (distribution-independent) SDE. 

In contrast to existing works such as \cite{Liu2022}, which rely on the Girsanov transformation and are restricted to additive noise and special drift structures, our method, based on an interpolated Euler-like sequence of classical SDEs, applies to general distribution-dependent stochastic systems driven by multiplicative noise, thereby highlighting its broader applicability.
\end{remark}

\begin{remark}\label{Remark 3.3}
The conclusion of Theorem~\ref{thm:om} can be extended, with suitable modifications, to non-Gaussian McKean--Vlasov SDEs, provided that $\int_{|\xi|<1} \xi\,\nu(d\xi) < \infty$ and that an interpolated Euler-like approximation scheme for the McKean--Vlasov dynamics is available; see, e.g., \cite{MR5056525} for related convergence results. Specifically, consider a Lévy-type McKean--Vlasov SDE on $\mathbb{R}^d$ of the form
\begin{equation}\label{Levy}
dX_t
=
b\bigl(t,X_t,\mu_t\bigr)\,dt
+
\sigma\,dW_t
+
\int_{|z|<1} z\,\tilde{N}(dt,dz),
\qquad X(0)=x_0,
\end{equation}
where $N(dt,dz)$ is a Poisson random measure on $\mathbb{R}^+\times\mathbb{R}^d$ with intensity measure $\nu(dz)\,dt$, independent of $W_t$, and $\tilde{N}(dt,dz)=N(dt,dz)-\nu(dz)\,dt$ denotes its compensated version. Under appropriate regularity assumptions, the Onsager--Machlup action functional takes the form, up to an additive constant,
\begin{equation}\label{eq:om-levy}
I(\phi,\dot{\phi})
=
\frac{1}{2}\int_0^T
\left[
\frac{\left|
\dot{\phi}_t
-
b(t,\phi_t,\delta_{\phi_t})
+
\int_{|\xi|<1}\xi\,\nu(d\xi)
\right|^2}{|\sigma|^2}
+
\operatorname{div} b(t,\phi_t,\delta_{\phi_t})
\right]dt.
\end{equation}

The study of transition mechanisms in such jump noise systems has significant potential for real-world applications \cite{Ditlevsen1999observation, Bianchi2010tempered, Bressloff2021}, particularly in biophysics, climate science, and economics.
\end{remark}

\subsection{Variational characterization of the most probable transition pathways}\label{sec:32}

Theorem~\ref{thm:om} enables the identification of the most probable transition pathway (MPTP) as a minimizer of the Onsager--Machlup functional. When transitions between specified states are of interest, the admissible path space can be defined as
$$
\mathcal{A} = \bigl\{ \phi \in C^2([0,T]; \mathbb{R}^d) \mid \phi(0) = z_0,\ \phi(T) = z_T \bigr\},
$$
where $z_0, z_T \in \mathbb{R}^d$ are given endpoints. (More generally, one may consider the union $\bigcup_{T>0} \mathcal{A}$ over varying time horizons.) In this framework, the goal is to determine whether there exists at least one function $\phi^* \in \mathcal{A}$ such that
$$
I(\phi^*, \dot{\phi}^*) = \min_{\phi \in \mathcal{A}} I(\phi, \dot{\phi}),
$$
in which case $\phi^*$ characterizes the MPTP connecting $z_0$ to $z_T$.

Furthermore, by the variational principle, any (local) minimizer of the action functional is a critical point and hence satisfies an associated Euler--Lagrange-type equation, together with the prescribed boundary conditions. In the present distribution-dependent setting, the Onsager--Machlup functional depends on the Dirac measure $\delta_{\phi_t}$ through the measure dependence of the coefficients. As a consequence, the corresponding Euler--Lagrange equation contains additional terms arising from differentiation with respect to the measure argument. These terms are understood in the sense of the $L$-derivative introduced in Definition~\ref{L-def}.

\begin{lemma}[Euler--Lagrange equation with Dirac measure dependence]
\label{lem:el}
Let
$$
\mathcal{L}:[0,T]\times \mathbb{R}^d\times \mathbb{R}^d\times \mathcal{P}_2(\mathbb{R}^d)\to \mathbb{R},
\qquad
(t,x,v,\mu)\mapsto \mathcal{L}(t,x,v,\mu),
$$
be continuously differentiable with respect to $x$ and $v$, and $L$-differentiable with respect to $\mu$ at every Dirac measure. Assume moreover that, along Dirac measures, the maps
$$
(t,x,v,z)\mapsto
\frac{\partial \mathcal L}{\partial x}(t,x,v,\delta_z),
\qquad
(t,x,v,z)\mapsto
\frac{\partial \mathcal L}{\partial v}(t,x,v,\delta_z),
$$
and
$$
(t,x,v,z)\mapsto
D_\mu\mathcal L(t,x,v,\delta_z)(z)
$$
are continuous on
$[0,T]\times\mathbb R^d\times\mathbb R^d\times\mathbb R^d$.
Here, for a Dirac measure $\delta_z$, the element
\(D_\mu\mathcal L(t,x,v,\delta_z)(\cdot)
\in L^2_{\delta_z}(\mathbb R^d;\mathbb R^d)\)
is canonically identified with its value at the support point,
namely \(D_\mu\mathcal L(t,x,v,\delta_z)(z)\in\mathbb R^d\).

For an admissible path $\phi\in\mathcal{A}$, consider the action functional
$$
J(\phi)=\int_0^T \mathcal{L}\bigl(t,\phi_t,\dot{\phi}_t,\delta_{\phi_t}\bigr)\,dt,
$$
with fixed endpoint conditions $\phi(0)=\phi_0$ and $\phi(T)=\phi_T$. If $\phi$ is a critical point of $J$, then $\phi$ satisfies
\begin{equation}\label{EL-dirac}
\frac{d}{dt}
\frac{\partial \mathcal{L}}{\partial v}
\bigl(t,\phi_t,\dot{\phi}_t,\delta_{\phi_t}\bigr)
=
\frac{\partial \mathcal{L}}{\partial x}
\bigl(t,\phi_t,\dot{\phi}_t,\delta_{\phi_t}\bigr)
+
D_\mu \mathcal{L}
\bigl(t,\phi_t,\dot{\phi}_t,\delta_{\phi_t}\bigr)(\phi_t).
\end{equation}
\end{lemma}

\begin{proof}
Let $\eta\in C_c^\infty((0,T);\mathbb{R}^d)$, and define the admissible variation
$$
\phi_t^\varepsilon=\phi_t+\varepsilon \eta_t,
$$
so that
$$
\dot\phi_t^\varepsilon
=
\dot\phi_t+\varepsilon\dot\eta_t,
\qquad
\delta_{\phi_t^\varepsilon}
=
\delta_{\phi_t+\varepsilon\eta_t}.
$$
Since $\eta$ has compact support in $(0,T)$, the perturbed path $\phi^\varepsilon$ satisfies the same endpoint conditions as $\phi$. Because $\phi$ is a critical point of $J$, we have
$$
\left.\frac{d}{d\varepsilon}J(\phi^\varepsilon)\right|_{\varepsilon=0}=0.
$$

For convenience, write (for fixed $t$)
$$
x=\phi_t,\qquad v=\dot\phi_t,\qquad a=\eta_t,\qquad b=\dot\eta_t,
$$
and define
$$
\widehat{\mathcal L}_t(\varepsilon)
:=
\mathcal L \bigl( t, x+\varepsilon a, v+\varepsilon b, \delta_{x+\varepsilon a} \bigr).
$$
We compute $\frac{d}{d\varepsilon}\widehat{\mathcal L}_t(\varepsilon)$ by splitting the increment:
$$
\begin{aligned}
\widehat{\mathcal L}_t(\varepsilon)-\widehat{\mathcal L}_t(0)
=
& \Bigl[ \mathcal L (t,x+\varepsilon a,v+\varepsilon b,\delta_{x+\varepsilon a}) - \mathcal L (t,x,v+\varepsilon b,\delta_{x+\varepsilon a}) \Bigr] \\
&+ \Bigl[ \mathcal L (t,x,v+\varepsilon b,\delta_{x+\varepsilon a}) - \mathcal L (t,x,v,\delta_{x+\varepsilon a}) \Bigr] \\
&+ \Bigl[ \mathcal L (t,x,v,\delta_{x+\varepsilon a}) - \mathcal L (t,x,v,\delta_x) \Bigr] \\
=:&\ \Delta_1+\Delta_2+\Delta_3.
\end{aligned}
$$
The first two terms are handled by ordinary differentiability in $x$ and $v$. Thus,
$$
\frac{\Delta_1}{\varepsilon}\to
\left\langle \frac{\partial \mathcal L}{\partial x} (t,x,v,\delta_x), a \right\rangle
\quad\text{and}\quad
\frac{\Delta_2}{\varepsilon}\to
\left\langle \frac{\partial \mathcal L}{\partial v} (t,x,v,\delta_x), b \right\rangle,
$$
as $\varepsilon \to 0$. It remains to compute the Dirac measure contribution. Observe that
$$
\delta_{x+\varepsilon a}
=
\delta_x\circ(\mathrm{Id}+\psi_\varepsilon)^{-1},
$$
where one may choose $\psi_\varepsilon\equiv \varepsilon a$, viewed as an element of
$L_{\delta_x}^2(\mathbb R^d;\mathbb R^d)$.
According to the definition of the $L$-derivative,
$$
\frac{\Delta_3}{\varepsilon}
\to
\left\langle D_\mu\mathcal L(t,x,v,\delta_x)(x), a \right\rangle,\;\;\text{as}\;\;\varepsilon\to 0.
$$
Combining the three limits, we obtain
\begin{equation*}
    \begin{aligned}
    \left. \frac{d}{d\varepsilon} J(\phi+\varepsilon\eta) \right|_{\varepsilon=0}
    =& \int_0^T
    \left\langle
    \frac{\partial \mathcal L}{\partial x}
    \bigl(t,\phi_t,\dot\phi_t,\delta_{\phi_t}\bigr)
    +
    D_\mu\mathcal L
    \bigl(t,\phi_t,\dot\phi_t,\delta_{\phi_t}\bigr)(\phi_t),
    \eta_t
    \right\rangle dt \\
    &+
    \int_0^T
    \left\langle
    \frac{\partial \mathcal L}{\partial v}
    \bigl(t,\phi_t,\dot\phi_t,\delta_{\phi_t}\bigr),
    \dot\eta_t
    \right\rangle dt.
    \end{aligned}
\end{equation*}
Integrating the second term on the right-hand side by parts and using
$\eta(0)=\eta(T)=0$, we obtain
$$
0
=
\int_0^T
\left\langle
\frac{\partial \mathcal L}{\partial x}
\bigl(t,\phi_t,\dot\phi_t,\delta_{\phi_t}\bigr)
+
D_\mu\mathcal L
\bigl(t,\phi_t,\dot\phi_t,\delta_{\phi_t}\bigr)(\phi_t)
-
\frac{d}{dt}
\frac{\partial \mathcal L}{\partial v}
\bigl(t,\phi_t,\dot\phi_t,\delta_{\phi_t}\bigr),
\eta_t
\right\rangle dt.
$$
Since $\eta$ is arbitrary, we conclude that
$$
\frac{d}{dt}
\frac{\partial \mathcal L}{\partial v}
\bigl(t,\phi_t,\dot\phi_t,\delta_{\phi_t}\bigr)
=
\frac{\partial \mathcal L}{\partial x}
\bigl(t,\phi_t,\dot\phi_t,\delta_{\phi_t}\bigr)
+
D_\mu\mathcal L
\bigl(t,\phi_t,\dot\phi_t,\delta_{\phi_t}\bigr)(\phi_t)
$$
in the sense of distributions on $(0,T)$. Under the stated continuity assumptions, the above identity holds pointwise on $(0,T)$, and hence extends to $[0,T]$ by continuity.
\end{proof}

With Lemma~\ref{lem:el} at hand, we now apply it to the Lagrangian associated with the Onsager--Machlup formula in Theorem~\ref{thm:om}. For a scalar coefficient \(f=f(t,x,\mu)\), define, along Dirac
measures,
\[
{\bf D}_i f(t,x,\delta_z)
:=
\frac{\partial f}{\partial x_i}(t,x,\delta_z)
+
\bigl[D_\mu f(t,x,\delta_z)(z)\bigr]_i,
\qquad i=1,\ldots,d.
\]
For vector-valued or matrix-valued coefficients, this definition is
understood componentwise. Thus, for a vector field \(g\), 
$({\bf D}g)_{ki}
=
{\bf D}_i g_k,$ 
and for a matrix-valued coefficient \(M\), \({\bf D}M\) is the
third-order tensor whose \(i\)-th slice is \({\bf D}_iM\).

\begin{theorem}[Variational characterization of MPTPs]\label{thm:om-EL}
Consider the Onsager--Machlup action functional
\[
I(\phi)
=
\int_0^T
\mathcal{OM}\bigl(t,\phi_t,\dot{\phi}_t,\delta_{\phi_t}\bigr)\,dt,
\]
where $\mathcal{OM}$ is given by \eqref{OMf} and satisfies the regularity assumptions of Lemma~\ref{lem:el}. Let $\phi\in C^2([0,T];\mathbb R^d)$ be a critical point of $I$ with $\phi(0)=z_0$ and $\phi(T)=z_T$. Then, for $t\in[0,T]$, $\phi$ satisfies
\begin{equation}\label{eq:EL-OM-explicit}
\begin{aligned}
\ddot{\phi}_t
&=
\partial_t \widetilde b
+
\Sigma\,({\bf D}\widetilde b)^\top \Sigma^{-1}\widetilde b
+
\Big[
{\bf D}\widetilde b
-
\Sigma({\bf D}\widetilde b)^\top\Sigma^{-1}
\Big]\dot{\phi}_t
\\
&\quad
+
\frac12\Sigma
\Big[
(\dot{\phi}_t-\widetilde b)^\top
({\bf D}\Sigma^{-1})
(\dot{\phi}_t-\widetilde b)
+
{\bf D}(\operatorname{div}\widetilde b)
-
\frac16{\bf D}R
\Big]
\\
&\quad
-
\Sigma
\left(
\partial_t \Sigma^{-1}
+
({\bf D}\Sigma^{-1})\dot{\phi}_t
\right)
(\dot{\phi}_t-\widetilde b).
\end{aligned}
\end{equation}
Here every occurrence of $\widetilde b$, $\Sigma$, $\Sigma^{-1}$, $R$, and their derivatives in \eqref{eq:EL-OM-explicit} is evaluated at $(t,\phi_t,\delta_{\phi_t})$.

In particular, when $d=1$ and $\sigma$ is constant, so that $\widetilde b=b$, $\Sigma=\sigma^2$, and $R=0$, the above equation reduces to
\begin{equation}\label{eq:EL-1d-constant}
\ddot{\phi}_t
=
\partial_t b
+
b\,{\bf D}b
+
\frac{\sigma^2}{2}\,{\bf D}(\partial_x b),
\end{equation}
or, equivalently,
\[
\ddot{\phi}_t
=
\partial_t b
+
b\,\partial_x b
+
b\,D_\mu b(t,\phi_t,\delta_{\phi_t})(\phi_t)
+
\frac{\sigma^2}{2}\partial_x^2 b
+
\frac{\sigma^2}{2}
D_\mu(\partial_x b)(t,\phi_t,\delta_{\phi_t})(\phi_t).
\]
\end{theorem}

\begin{proof}
We regard the Onsager--Machlup integrand as a Lagrangian:
\[
\mathcal{OM}(t,x,v,\mu)
=
\frac12
\left[
(v-\widetilde b)^\top
\Sigma^{-1}
(v-\widetilde b)
+
\operatorname{div}\widetilde b
-
\frac16 R
\right],
\]
where all coefficients are evaluated at $(t,x,\mu)$. Applying Lemma~\ref{lem:el} and using the notation ${\bf D}$, any critical point satisfies
\begin{equation}\label{el-om}
\frac{d}{dt}
\frac{\partial \mathcal{OM}}{\partial v}
\bigl(t,\phi_t,\dot\phi_t,\delta_{\phi_t}\bigr)
=
{\bf D}\mathcal{OM}
\bigl(t,\phi_t,\dot\phi_t,\delta_{\phi_t}\bigr).
\end{equation}

On the one hand, since ${\bf D}$ differentiates both the state variable and the Dirac support, we have
\begin{equation}\label{el-om-1}
{\bf D}\mathcal{OM}
=
-({\bf D}\widetilde b)^\top\Sigma^{-1}F_t
+
\frac12
F_t^\top({\bf D}\Sigma^{-1})F_t
+
\frac12{\bf D}(\operatorname{div}\widetilde b)
-
\frac1{12}{\bf D}R,
\end{equation}
where
\[
F_t:=\dot{\phi}_t-\widetilde b(t,\phi_t,\delta_{\phi_t}).
\]
On the other hand,
\[
\frac{\partial \mathcal{OM}}{\partial v}
=
\Sigma^{-1}F_t.
\]
Applying the product rule along the curve $t\mapsto (t,\phi_t,\delta_{\phi_t})$, we obtain
\begin{align}\label{el-om-2}
\frac{d}{dt}
\frac{\partial \mathcal{OM}}{\partial v}
=
\frac{d}{dt}\bigl(\Sigma^{-1}F_t\bigr)
=
\left(
\frac{\partial \Sigma^{-1}}{\partial t}
+
({\bf D}\Sigma^{-1})\dot{\phi}_t
\right)F_t
+
\Sigma^{-1}
\left[
\ddot{\phi}_t
-
\frac{\partial \widetilde b}{\partial t}
-
({\bf D}\widetilde b)\dot{\phi}_t
\right].
\end{align}

Substituting \eqref{el-om-1} and \eqref{el-om-2} into \eqref{el-om}, then multiplying by $\Sigma$ on both sides and solving for $\ddot{\phi}_t$, we get
\begin{align*}
\ddot{\phi}_t
&=
\frac{\partial \widetilde b}{\partial t}
+
({\bf D}\widetilde b)\dot{\phi}_t
-
\Sigma
\left(
\frac{\partial \Sigma^{-1}}{\partial t}
+
({\bf D}\Sigma^{-1})\dot{\phi}_t
\right)F_t
\\
&\quad
-
\Sigma({\bf D}\widetilde b)^\top\Sigma^{-1}F_t
+
\frac12\Sigma
\left[
F_t^\top({\bf D}\Sigma^{-1})F_t
+
{\bf D}(\operatorname{div}\widetilde b)
-
\frac16{\bf D}R
\right].
\end{align*}
Using $F_t=\dot{\phi}_t-\widetilde b$, the drift-gradient terms become
\[
({\bf D}\widetilde b)\dot{\phi}_t
-
\Sigma({\bf D}\widetilde b)^\top\Sigma^{-1}F_t
=
\Sigma({\bf D}\widetilde b)^\top\Sigma^{-1}\widetilde b
+
\Big[
{\bf D}\widetilde b
-
\Sigma({\bf D}\widetilde b)^\top\Sigma^{-1}
\Big]\dot{\phi}_t.
\]
This proves \eqref{eq:EL-OM-explicit}.

Finally, suppose $d=1$ and $\sigma$ is constant. Then $\Sigma=\sigma^2$, $\widetilde b=b$, $R=0$, and all terms containing ${\bf D}\Sigma^{-1}$ or $\partial_t\Sigma^{-1}$ vanish. In the scalar case, the identity
\[
{\bf D}b-\Sigma({\bf D}b)\Sigma^{-1}=0
\]
yields \eqref{eq:EL-1d-constant}.
\end{proof}

\subsection{Proof of Theorem \ref{thm:om}}\label{sec:33}

We are now ready to demonstrate how to derive the Onsager--Machlup functional for the McKean--Vlasov SDE \eqref{Main-equation}. Based on Definition~\ref{def:om}, we aim to estimate the probability that the solution process $X$ of \eqref{Main-equation} lies within a tube of radius $\varepsilon$ around a reference path $\phi$, i.e., $\mathds{P}\bigl(\{\|X - \phi\| \leq \varepsilon\}\bigr)$. The derivation is divided into the following three steps.

{\bf Step 1: Convergence of the approximating solutions.}\par
By Proposition~\ref{thm:convergence}, the Euler-type sequence $\{X^{(n)}\}_{n\ge 1}$ converges to the unique strong solution $X$ of the McKean–Vlasov equation \eqref{Main-equation} in the $L^2(\Omega;C([0,T];\mathbb{R}^d))$-sense:
$$
\lim_{n\to\infty} \Bigl(\mathbb{E}\Bigl[\sup_{0\le t\le T}|X^{(n)}(t)-X(t)|^2\Bigr]\Bigr)^{1/2}=0.
$$
Consequently, there exists a subsequence (still denoted by $\{X^{(n)}\}$ for simplicity) such that, for $\mathds{P}$-almost every $\omega \in \Omega$,
$$
\|X^{(n)}(\omega) -X(\omega)\|=\sup_{t\in[0,T]}|X^{(n)}_t(\omega) -X_t(\omega)|\to 0.
$$
For any $\delta>0$, define
$$
A_n:=\bigcap_{m=n}^{\infty}\left\{\|X^{(m)}-X\|\leq \delta\right\}.
$$
Then $\{A_n\}_{n\geq 1}$ is an increasing sequence and $\mathds{P}\left(\bigcup_{n=1}^{\infty}A_n\right)=1$.

{\bf Step 2: Asymptotic probability for the Euler-type approximation sequence.}\par
Recall the Euler-type approximation sequence $\{X^{(n)}\}_{n \geq 1}$ constructed in Section~\ref{subsec:euler-sequence}. For each fixed $n$, the process $X^{(n)}$ satisfies the classical (distribution-free) SDE
\begin{equation}\label{eq:euler-approx-R}
dX^{(n)}_t = b^{(n)}(t, X^{(n)}_t) \, dt + \sigma^{(n)}(t, X^{(n)}_t) \, dW_t,
\end{equation}
where, for convenience, we introduce the shorthand notation
$$
b^{(n)}(t, X_t^{(n)}) = b\bigl(t, X^{(n)}_t, \mu_{\kappa_n(t)}^{(n)}\bigr), \qquad
\sigma^{(n)}(t, X_t^{(n)}) = \sigma\bigl(t, X^{(n)}_t, \mu_{\kappa_n(t)}^{(n)}\bigr).
$$
Evaluated at the reference path $\phi$, we have
$$
b^{(n)}(t,\phi_t)=b\bigl(t,\phi_t,\delta_{\phi_{\kappa_n(t)}}\bigr),\qquad
\sigma^{(n)}(t,\phi_t)=\sigma\bigl(t,\phi_t,\delta_{\phi_{\kappa_n(t)}}\bigr),
$$
with $\delta_{\phi_{\kappa_n(t)}}$ denoting the Dirac measure at $\phi_{\kappa_n(t)}$. The associated Onsager--Machlup functional $I^{(n)}(\phi,\dot{\phi})$ is then given by the classical geometric expression \cite{Zeitouni1988}:
\begin{align}\label{OM-classical}
I^{(n)}(\phi, \dot{\phi})
=
\frac{1}{2} \int_0^T
\Bigl(&
\bigl( \dot{\phi}_t - \widetilde{b}^{(n)}(t, \phi_t) \bigr)^\top
(\Sigma^{(n)}(t, \phi_t))^{-1}
\bigl( \dot{\phi}_t - \widetilde{b}^{(n)}(t, \phi_t) \bigr)\notag\\
&+\operatorname{div} \widetilde{b}^{(n)}(t, \phi_t)
-
\frac{1}{6} R^{(n)}(t, \phi_t)
\Bigr)\,dt,
\end{align}
where $\Sigma^{(n)}(t, \phi_t)=\sigma^{(n)}(t, \phi_t)\sigma^{(n)}(t, \phi_t)^\top$, $R^{(n)}(t, \phi_t)$ denotes the scalar curvature associated with the Riemannian metric induced by $(\Sigma^{(n)})^{-1}$, and the modified drift $\widetilde{b}^{(n)}$ is defined by
\begin{equation}\label{b}
\widetilde{b}^{(n),i}
=
b^{(n),i}
-
\frac{1}{2}\sum_{l,j} (\Sigma^{(n)})^{lj}\Gamma_{lj}^{(n),i},
\end{equation}
with $\Gamma_{lj}^{(n),i}$ the corresponding Christoffel symbols. The Riemannian divergence of $\widetilde{b}^{(n)}$ is given by
$$
\operatorname{div} \widetilde{b}^{(n)}
=
\frac{1}{\sqrt{\det\bigl((\Sigma^{(n)})^{-1}\bigr)}}
\sum_i
\frac{\partial}{\partial x^i}
\left(
\widetilde{b}^{(n),i}
\sqrt{\det\bigl((\Sigma^{(n)})^{-1}\bigr)}
\right).
$$
For the classical SDE \eqref{eq:euler-approx-R}, the Onsager--Machlup theory yields the asymptotic relation that, for sufficiently small $\varepsilon>0$,
\begin{equation}\label{eq:classical-asymptotic}
\mathds{P}\bigl(\{\|X^{(n)} - \phi\| \leq \varepsilon\}\bigr) \propto C(\varepsilon,T) \exp\bigl\{ -I^{(n)}(\phi, \dot{\phi}) \bigr\},
\end{equation}
where $\propto$ denotes asymptotic proportionality.

{\bf Step 3: Estimation of the tube probability.}\par
We combine Steps 1 and 2 to obtain the desired asymptotics for the McKean--Vlasov process $X$. Fix a reference path $\phi$ and a sufficiently small $\varepsilon>0$.
Using the almost‑sure convergence from Step 1, we have
\begin{align*}
\mathds{P}\bigl(\{\|X-\phi\|\leq \varepsilon\}\bigr)
&=
\mathds{P}\left(\{\|X-\phi\|\leq \varepsilon\}\cap \bigcup_{n=1}^{\infty}A_n\right) \notag \\
&=
\mathds{P}\left(\bigcup_{n=1}^{\infty}\bigl(\{\|X-\phi\|\leq \varepsilon\}\cap A_n\bigr)\right) \notag \\
&=
\lim_{n\to\infty}
\mathds{P}\left(\{\|X-\phi\|\leq \varepsilon\}\cap A_n\right),
\end{align*}
where the last equality follows from the continuity from below of probability measures. Since
$$
A_n \subset \left\{\|X^{(n)}-X\|\leq \delta\right\},
$$
for any $\delta>0$, we obtain
\begin{align*}
\mathds{P}\bigl(\{\|X-\phi\|\leq \varepsilon\}\bigr)
&\leq
\lim_{n\to\infty}
\mathds{P}\left(\{\|X-\phi\|\leq \varepsilon\}\cap \{\|X^{(n)}-X\|\leq \delta\}\right) \\
&\leq
\lim_{n\to\infty}
\mathds{P}\left(\{\|X^{(n)}-\phi\|\leq \varepsilon+\delta\}\right),
\end{align*}
where the last inequality follows from the triangle inequality. By \eqref{eq:classical-asymptotic}, we infer that 
\begin{align}\label{UB}
\mathds{P}\bigl(\{\|X-\phi\|\leq \varepsilon\}\bigr)
\leq
C(\varepsilon+\delta,T)\lim_{n\to\infty}\exp\bigl\{-I^{(n)}(\phi,\dot{\phi})\bigr\}.
\end{align}
Now note that the Wasserstein distance satisfies
$$
\lim_{n\to\infty}\sup_{0\leq t\leq T}W_2^2\bigl(\delta_{\phi_{\kappa_n(t)}},\delta_{\phi_t}\bigr)
\leq
\lim_{n\to\infty}\sup_{0\leq t\leq T}|\phi_t-\phi_{\kappa_n(t)}|^2
\leq
\lim_{n\to\infty}h_n
=0.
$$
By Assumption~\ref{assump:main}(3), it follows that for each $t\in[0,T]$,
$$
b^{(n)}(t,\phi_t)
\to
b\bigl(t,\phi_t,\delta_{\phi_t}\bigr),
\qquad
\sigma^{(n)}(t,\phi_t)
\to
\sigma\bigl(t,\phi_t,\delta_{\phi_t}\bigr),
$$
as $n\to\infty$. Consequently, by the continuity of the coefficients in \eqref{OM-classical} and the dominated convergence theorem, we have
$$
I^{(n)}(\phi,\dot{\phi})
\to
I(\phi,\dot{\phi}),
$$
where 
\begin{align}
    I(\phi, \dot{\phi}) = \frac{1}{2} \int_0^T \Bigl(& \bigl( \dot{\phi}_t - \widetilde{b}(t, \phi_t, \delta_{\phi_t}) \bigr)^\top \Sigma^{-1}(t, \phi_t, \delta_{\phi_t}) \bigl( \dot{\phi}_t - \widetilde{b}(t, \phi_t, \delta_{\phi_t}) \bigr) \notag\\
    &+ \operatorname{div} \widetilde{b}(t, \phi_t, \delta_{\phi_t}) - \frac{1}{6} R(t, \phi_t) \Bigr) \, dt.
\end{align}
Letting $\delta\to 0$ in \eqref{UB} and using the continuity of $C(\varepsilon,T)$ and the exponential function, we derive the upper bound
$$
\mathds{P}\bigl(\{\|X-\phi\|\leq \varepsilon\}\bigr)
\leq
C(\varepsilon,T)\exp\bigl\{-I(\phi,\dot{\phi})\bigr\},
\qquad \varepsilon\to 0.
$$

For the lower bound, we proceed similarly. 
For any $\delta>0$,
$$
\{\|X^{(n)}-\phi\|\leq \varepsilon-\delta\}\cap \{\|X^{(n)}-X\|\leq \delta\}
\subset
\{\|X-\phi\|\leq \varepsilon\}.
$$
Taking probabilities and using the convergence from Step 1,
\begin{align}\label{LB}
\mathbb{P}\bigl(\{\|X-\phi\|\le\varepsilon\}\bigr)
&\ge\lim_{n\to\infty}\mathbb{P}\bigl(\{\|X^{(n)}-\phi\|\le\varepsilon-\delta\}\cap\{\|X^{(n)}-X\|\le\delta\}\bigr)\notag\\
&\ge\lim_{n\to\infty}\Bigl(\mathbb{P}\bigl(\{\|X^{(n)}-\phi\|\le\varepsilon-\delta\}\bigr)-\mathbb{P}\bigl(\{\|X^{(n)}-X\|>\delta\}\bigr)\Bigr)\notag\\
&=\lim_{n\to\infty}\mathbb{P}\bigl(\{\|X^{(n)}-\phi\|\le\varepsilon-\delta\}\bigr)\quad\bigl(\mathbb{P}(\{\|X^{(n)}-X\|>\delta\})\to0\bigr)\notag\\
&\propto C(\varepsilon-\delta,T)\,\lim_{n\to\infty}\exp\bigl\{-I^{(n)}(\phi,\dot\phi)\bigr\}.
\end{align}
Letting $\delta\to 0$ and using again the continuity of $C(\cdot,T)$ and the exponential function, we obtain the lower bound
$$
\mathds{P}\bigl(\{\|X-\phi\|\leq \varepsilon\}\bigr)
\geq
C(\varepsilon,T)\exp\bigl\{-I(\phi,\dot{\phi})\bigr\},
\qquad \varepsilon\to 0.
$$

Combining the upper and lower bounds, we conclude that
$$
\mathds{P}\bigl(\{\|X-\phi\|\leq \varepsilon\}\bigr)
\propto
C(\varepsilon,T)\exp\bigl\{-I(\phi,\dot{\phi})\bigr\},
\qquad \varepsilon\to 0,
$$
which completes the derivation of the Onsager--Machlup functional for the McKean--Vlasov SDE \eqref{Main-equation}.


\renewcommand{\theequation}{\thesection.\arabic{equation}}
\setcounter{equation}{0}

\section{An illustrative example}\label{sec:4}
We now present a simple concrete example to illustrate our theoretical results.

Consider the scalar McKean--Vlasov SDE
\begin{equation}\label{eq:example2}
dX_t = \bigl( X_t - X_t^3 + \mathbb{E}[X_t] \bigr) dt
+ \sigma \, dB_t, \qquad X_0 = 1,
\end{equation}
where $\sigma>0$ is a constant. 
The drift coefficient can be written as $b(x, \mu) = x - x^3 + \int_{\mathbb{R}} y \, \mu(dy)$. In the absence of the interaction term (i.e., setting the expectation to zero), the deterministic system $\dot{x} = x - x^3$ has three equilibrium points: $x = -1$, $x = 0$, and $x = 1$, among which $x = 0$ is unstable while $x = \pm 1$ are stable. These two stable equilibria correspond to the metastable states of the stochastic system \eqref{eq:example2}.

A direct verification shows that the drift coefficient satisfies  Assumption~\ref{assump:main}; moreover, the
diffusion coefficient is constant, with $\Sigma=\sigma^2$. Therefore, by
Theorem~\ref{thm:om} together with Remark~\ref{rem:generality}, the
Onsager--Machlup functional associated with \eqref{eq:example2} is well
defined. In the present one-dimensional setting, the general OM functional
reduces to
\[
I(\phi,\dot\phi)
=
\frac12\int_0^T
\left[
\frac{\bigl(\dot\phi_t-b(\phi_t,\delta_{\phi_t})\bigr)^2}{\sigma^2}
+
\operatorname{div} b(\phi_t,\delta_{\phi_t})
\right]dt.
\]
For a path $\phi$ satisfying the boundary conditions $\phi_0=1$ and
$\phi_T=-1$, we have
\[
b(\phi_t,\delta_{\phi_t})
=
\phi_t-\phi_t^3+\int_{\mathbb R}y\,\delta_{\phi_t}(dy)
=
2\phi_t-\phi_t^3.
\]
Consequently, the action functional becomes
\begin{equation}\label{eq:om-example2}
I(\phi,\dot{\phi})
=
\frac{1}{2} \int_0^T
\left(
\frac{ \bigl( \dot{\phi}_t - (2\phi_t - \phi_t^3) \bigr)^2 }{\sigma^2}
+ 1 - 3\phi_t^2
\right) dt.
\end{equation}

Next, applying Theorem~\ref{thm:om-EL} together with the one-dimensional
constant-diffusion formula \eqref{eq:EL-1d-constant}, and using
\[
D_\mu b(x,\mu)(z)=1,
\qquad
\frac{\partial b}{\partial x}(x,\mu)=1-3x^2,
\]
we obtain
\begin{equation}\label{eq:el-example2}
\ddot{\phi}_t
=
(2-3\phi_t^2)(2\phi_t-\phi_t^3)
-
3\sigma^2\phi_t,
\end{equation}
with $\phi_0=1$ and $\phi_T=-1$. This coincides with the Euler--Lagrange
equation obtained by directly applying the standard first variation to the
reduced functional \eqref{eq:om-example2}.

The most probable transition path (MPTP) $\phi^*$ is then obtained by
minimizing the OM action functional \eqref{eq:om-example2}, equivalently by
solving the two-point boundary value problem \eqref{eq:el-example2}.
Numerical solutions for $T=1$ with $\sigma=1$ and $\sigma=1.5$ are shown in
Figure~\ref{fig:path}. In both cases, the simulated sample paths concentrate
around the MPTP, shown as the bold red curve. This confirms that the
minimizer of the OM functional captures the most probable transition
behavior of the system. Moreover, as the noise intensity increases from
$\sigma=1$ to $\sigma=1.5$, transitions between the two metastable states
become more likely, which is consistent with physical intuition.

\begin{figure}[t]
 \centering
 \subfigure[]{
 \label{fig1a}
 \includegraphics[width=0.45\textwidth]{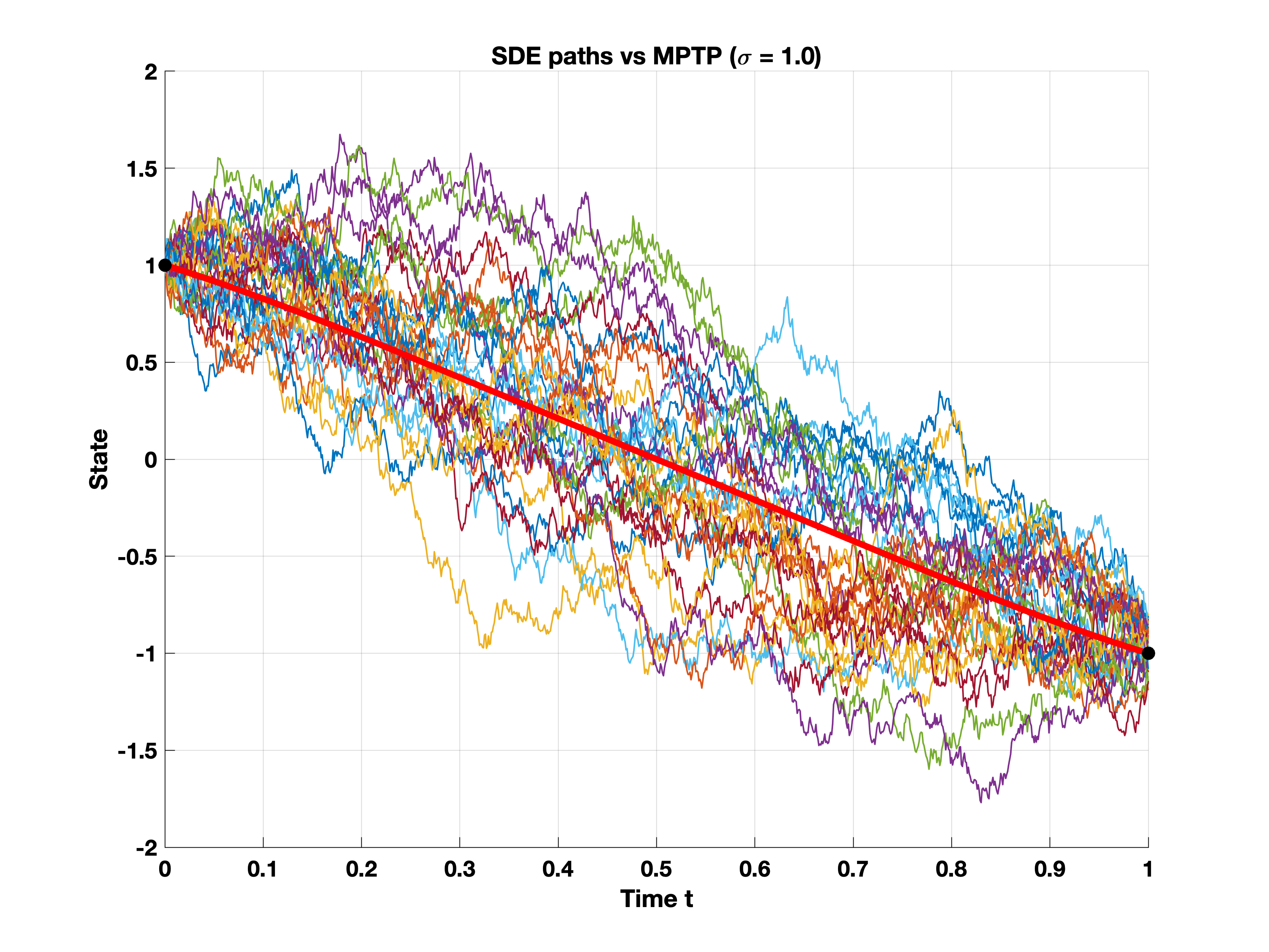}
 }
 \subfigure[]{
 \label{fig1b}
 \includegraphics[width=0.45\textwidth]{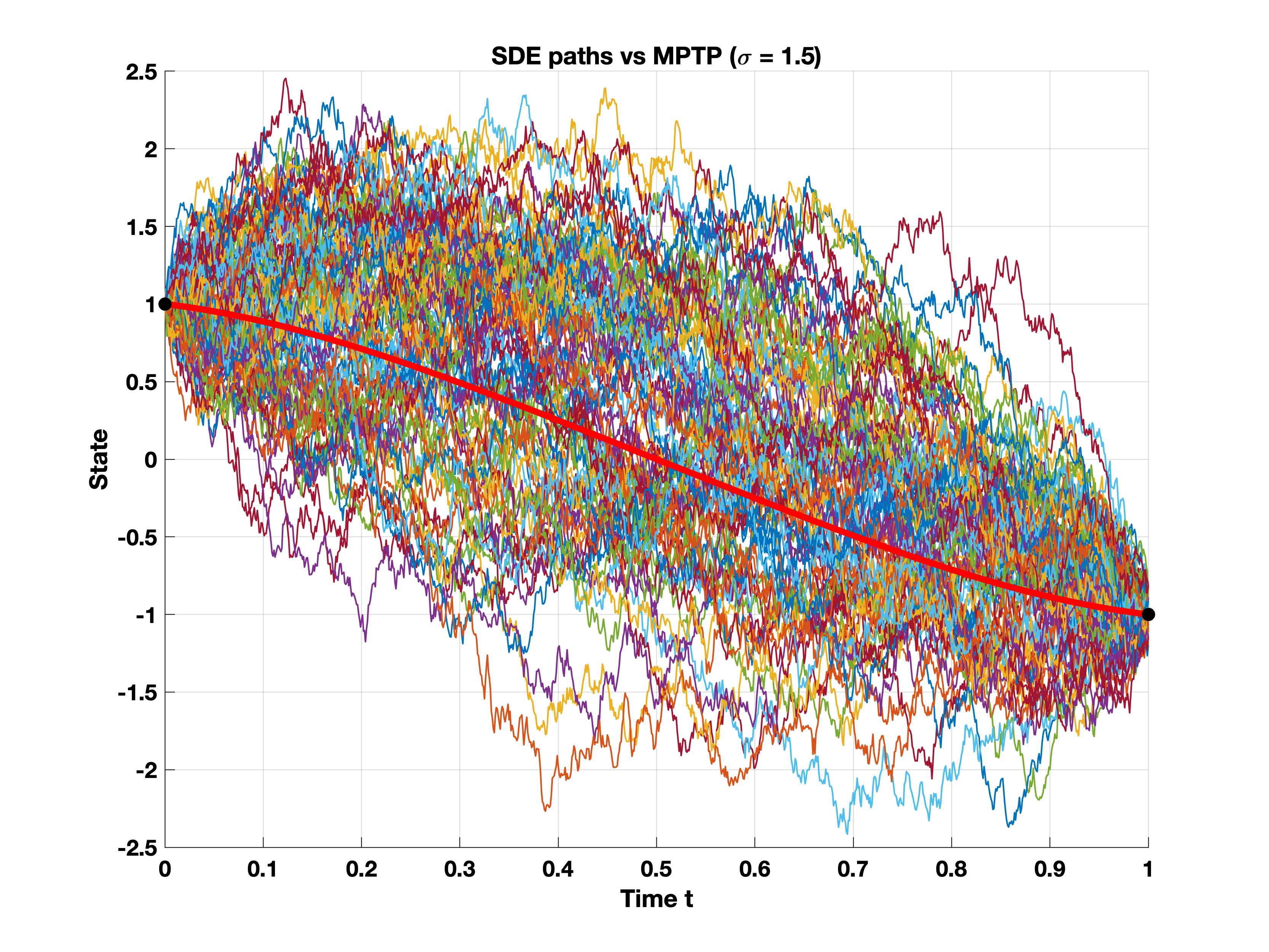}
 }
  \vspace{-0.5cm}  \caption{ (Color online) (a) The pattern of different sample paths and the MPTP (bold red line) when $\sigma=1$. (b) The pattern of different sample paths and the MPTP (bold red line) when $\sigma=1.5$.}
 \label{fig:path}
 \end{figure}

\section{Conclusion and Discussion}\label{sec:5}

In this work, we derived the Onsager--Machlup functional for a class of McKean--Vlasov SDEs under locally Lipschitz conditions in the state variable, allowing for a possibly superlinearly growing drift. The central idea is to approximate the original distribution-dependent system by an Euler-type sequence of classical, distribution-free SDEs, and then to pass to the limit by exploiting the classical Onsager--Machlup theory for each approximating system. This provides a novel and powerful strategy for studying distribution-dependent stochastic dynamics through approximation by classical diffusions. By viewing the Onsager--Machlup functional as a Lagrangian, we further obtained a variational characterization of the most probable transition path (MPTP) via the corresponding Euler--Lagrange equation with fixed endpoint conditions. The illustrative example of a particle in a double-well potential confirms that the resulting minimizer indeed captures the expected transition behavior.

The present framework is sufficiently general to suggest several further developments. In particular, with suitable modifications, the method can be extended to non-Gaussian McKean--Vlasov SDEs driven by Lévy noise, provided that $\int_{|\xi|<1}\xi\,\nu(d\xi)<\infty$ and that the convergence of the associated interpolated Euler-type approximation is available. It would also be of interest to investigate sharper regularity assumptions, higher-dimensional examples, and numerical schemes for computing the associated MPTPs in more general distribution-dependent settings.

\section*{Acknowledgements}
The research of YC was partially supported by 
Shaanxi Fundamental Science Research Project for Mathematics and Physics                                               (Grant No.25JSQ042),
by the Fundamental Research Funds for the Central Universities xzy012025071, by the NSFC grant 12101484. The research of PW was partially supported by the Natural Science Foundation of Jiangsu Province (Grant No.BK20251330) and the Jiangsu Provincial Scientific Research Center of Applied Mathematics (Grant No.BK20233002).

\section*{Data availability statement}
No new data were created or analysed in this study.

\section*{Conflict of interest}
The authors declare that they have no competing interests.

\section*{ORCID iD}
Ying Chao \orcidlink{0000-0002-0132-2680} 0000-0002-0132-2680\par
Pingyuan Wei \orcidlink{0000-0001-8446-538X} 0000-0001-8446-538X




\bibliographystyle{alpha}
\bibliography{YC_Refs}
\end{document}